\documentclass{article}
\usepackage[top=3cm, bottom=2.5cm, inner=2.5cm, outer=2.5cm]{geometry}

\usepackage[utf8]{inputenc}
\usepackage{lmodern}
\usepackage[T1]{fontenc}

\usepackage{amsmath, amssymb, amsthm, amsfonts}
\usepackage[english]{babel}
\usepackage{bbm}
\usepackage[retainorgcmds]{IEEEtrantools}
\usepackage{thm-restate}

\theoremstyle{plain}
\newtheorem{theorem}{Theorem}
\newtheorem{lemma}[theorem]{Lemma}

\newtheorem{claim}[theorem]{Claim}
\newtheorem{conjecture}[theorem]{Conjecture}
\newtheorem{question}[theorem]{Question}

\theoremstyle{definition}
\newtheorem{definition}[theorem]{Definition}
\newtheorem{example}[theorem]{Example}

\theoremstyle{remark}
\newtheorem{remark}[theorem]{Remark}

\DeclareMathOperator{\supp}{supp}

\DeclareMathOperator{\poly}{poly}
\DeclareMathOperator{\same}{same}
\DeclareMathOperator{\cycle}{cycle}
\DeclareMathOperator{\diag}{diag}

\newcommand{\bbN}{\mathbb{N}}
\newcommand{\cE}{\mathcal{E}}
\newcommand{\cP}{\mathcal{P}}
\newcommand{\eps}{\varepsilon}

\usepackage[pdftex]{hyperref}

\author{
Jan Hązła\thanks{EPFL, {\tt jan.hazla@gmail.com}.
The author was supported by awards ONR N00014-16-1-2227,  
NSF CCF-1665252 and ARO MURI W911NF-12-1-0509.}
}

\title{On arithmetic progressions in symmetric sets in
finite field model}

\date{}

\begin{document}

\maketitle

\begin{abstract}
  We consider two problems regarding arithmetic progressions in
  symmetric sets in the finite field (product space) model.

  First, we show that a symmetric set
  $S \subseteq \mathbb{Z}_q^n$ containing $|S| = \mu \cdot q^n$ elements
  must contain at least $\delta(q, \mu) \cdot q^n \cdot 2^n$ arithmetic
  progressions $x, x+d, \ldots, x+(q-1)\cdot d$ such that the difference $d$
  is restricted to lie in $\{0,1\}^n$.

  Second, we show that for prime $p$ a symmetric set $S\subseteq\mathbb{F}_p^n$
  with $|S|=\mu\cdot p^n$ elements contains at least
  $\mu^{C(p)}\cdot p^{2n}$ arithmetic progressions of length $p$.
  This establishes that the qualitative behavior of longer arithmetic
  progressions in symmetric sets is the same as for progressions of length
  three.
\end{abstract}

\section{Introduction}

In this paper we consider problems in the finite field model
in additive combinatorics. This model has been a fruitful area of research,
originally considered as a ``playground'' for classical problems 
over integers, but subsequently becoming a source of many results that
are interesting on their own. The reader can consult
two surveys \cite{Gre05, Wol15}
that are removed in time by ten years.

The most well-known problem in this setting concerns arithmetic progressions:
Given a subset $S \subseteq \mathbb{Z}_q^n$
with density $\mu(S) := |S|/q^n$, what are the bounds on the number
of arithmetic progressions of length $k$ contained in $S$? The 
case $q = k = 3$ is called the \emph{capset problem}. There,
it has long been known
\cite{Rot53, Mes95}
that any subset of $\mathbb{F}_3^n$ of constant density must contain an
arithmetic progression
of length three for large enough $n$. Subsequent improvements
culminating in recent breakthrough applying the polynomial method
\cite{CLP17, EG17}
establish that (contrary to the integer case as evidenced by the Behrend's
construction) the largest progression-free set in $\mathbb{F}_3^n$
has density that is exponentially small in $n$. It is also well known
that this last statement is equivalent to the following: There exists a
constant $C > 0$ such that every set $S \subseteq \mathbb{F}_3^n$ with
density $\mu$ contains at least $\mu^C \cdot 9^n$ arithmetic
progressions of length three (including among $9^n$ progressions degenerate
ones with difference zero).

As for longer progressions,
while it is known (for example using the density
Hales-Jewett theorem~\cite{FK91}) that dense subsets of
$\mathbb{F}_p^n$ contain a dense proportion of progressions of any length $k$,
the quantitative bounds are quite weak with the
exception of progressions of length four
(see \cite{GT12}), where it has been established by Green and Tao that
a set of density $\mu$ contains at least an $\exp(-\poly(1/\mu))$ proportion
of all progressions.

We present a result that achieves a $\mu^C$ type of bound for arbitrarily
long progressions, at the expense of restricting ourselves to
\emph{symmetric} sets: Subsets $S\subseteq \mathbb{F}_p^n$ where
membership $x \in S$ is invariant under permutations of coordinates.
More formally,
for $x \in \mathbb{Z}_q^n$ and $a \in \mathbb{Z}_q$ we define
the \emph{weight} $w_a(x) := \left|\left\{i \in [n]: x_i = a\right\}\right|$.
We say that $S \subseteq \mathbb{Z}_q^n$ is symmetric if membership
$x \in S$ depends only on the weight tuple
$(w_0(x), \ldots, w_{q-1}(x))$.

In fact, we prove a more general removal lemma. In the following we find
it useful to frame our statements in terms of probabilities. For that purpose,
let $X^{(1)},\ldots,X^{(p)} \in \mathbb{F}_p^n$ be random variables
representing a uniformly random arithmetic progression of length $p$,
i.e., $X^{(j)} = (X_1^{(j)}, \ldots, X_n^{(j)})$ for $j=1,\ldots,p$, and
tuples $(X_i^{(1)},\ldots,X_i^{(p)})$ for $i=1,\ldots,n$ are independent
and distributed uniformly among $p^2$ possible progressions
$x,x+d,\ldots,x+(p-1)d$ for $x,d\in\mathbb{F}_p$. 

\begin{theorem}\label{thm:poly}
  Let $p \ge 3$ be prime and $0 < \mu < 1$. 
  There exist $n_0$ and $C > 0$ such
  that for all $n \ge n_0$, if $S^{(1)},\ldots,S^{(p)}$
  are symmetric subsets of $\mathbb{F}_p^n$ satisfying
  \begin{align}\label{eq:22}
    \Pr\left[
    X^{(1)}\in S^{(1)}\land\ldots\land X^{(p)}\in S^{(p)}
    \right] < \mu^C \; ,
  \end{align}
  then there exist symmetric sets
  $S'^{(1)},\ldots,S'^{(p)}\subseteq\mathbb{F}_p^n$, each of density
  at most $\mu$, such that letting $T^{(j)} := S^{(j)}\setminus S'^{(j)}$ we have
  \begin{align*}
    \Pr\left[
    X^{(1)}\in T^{(1)}\land\ldots\land X^{(p)}\in T^{(p)}
    \right]=0 \; .
  \end{align*}
\end{theorem}

Taking $S^{(1)}=\cdots=S^{(p)}=S$ and noting that due to trivial progressions
with difference $d=0^n$ the probability
$\Pr[X^{(1)},\ldots,X^{(p)}\in S] = 0$ if and only if $S$ is empty,
it follows that a symmetric set of density $\mu$ contains at least
$(\mu/p)^C \cdot p^{2n}$ progressions.\footnote{We note that the question of existence of an arithmetic progression
  in a symmetric set has a positive answer for a less interesting reason:
  If a symmetric set $S$ contains an element $x$ with all weights $w_a(x)$ non-zero,
  then it is easy to find a progression in $S$ consisting only of permutations
  of coordinates of $x$.
  }
We remark that Theorem~\ref{thm:poly} has a weakness in that its conclusion
holds only for large enough $n$ after fixing the density $\mu$. The technical
reason is that we apply a version of the local limit theorem without an
explicit error bound. We do not attempt to fix this deficiency in this work.

The second problem we consider concerns arithmetic progressions
in $\mathbb{Z}_q^n$ with the difference restricted to lie
in $\{0,1\}^n$. Accordingly, we call them \emph{restricted progressions}.
Again, an application of the density Hales-Jewett theorem
establishes that a dense set $S\subseteq \mathbb{Z}_q^n$ contains
a non-trivial restricted progression of length $q$ for large enough $n$.
However, the author is
not aware of a proof that a dense set contains a dense fraction of
all such progressions.

Our second result is a removal lemma for symmetric sets with respect
to restricted progressions:
\begin{theorem}
  \label{thm:removal}
  Let $q \ge 3$ and $\mu > 0$. There exists $\delta := \delta(q, \mu) > 0$
  such that for every tuple of symmetric sets
  $S^{(1)}, \ldots, S^{(q)} \subseteq \mathbb{Z}_q^n$ the following holds:
  Letting $X^{(1)},\ldots,X^{(q)}$ be a random
  arithmetic progression in $\mathbb{Z}_q^n$ with a difference restricted
  to $\{0,1\}^n$, if
  \begin{align*}
    \Pr\left[ X^{(1)} \in S^{(1)}\land\cdots\land X^{(q)}\in S^{(q)}\right]
    <\delta\; ,
  \end{align*}
  then there exists a symmetric set $S'$ with density at most $\mu$ such that
  for $T^{(j)} := S^{(j)}\setminus S'$ we have
  \begin{align*}
    \Pr\left[ X^{(1)} \in T^{(1)}\land\cdots\land X^{(q)}\in T^{(q)}\right]
    =0\; .
  \end{align*}
\end{theorem}

Similar as in the case of Theorem~\ref{thm:poly}, it follows that
a symmetric set $S$ of density $\mu$ contains a dense fraction of all
restricted progressions.

\subsection{Proof idea}

The proofs of Theorems~\ref{thm:poly} and~\ref{thm:removal} are applications of
the same technique, proceeding in two stages. First,
we use a local central limit theorem to show that those theorems are implied
(in fact, equivalent to) certain additive combinatorial statements
over the integers (more precisely, over $\mathbb{Z}^{q-1}$). Since
the membership $x\in S$ depends only on the weight tuple
$(w_0(x),\ldots,w_{q-2}(x))$ (we omit $w_{q-1}(x)$, since, knowing $n$,
it can be inferred from the other components), we can think in terms of weight tuples
in $\mathbb{Z}^{q-1}$ rather than vectors in $\mathbb{Z}_q^n$. Furthermore,
the CLT argument shows that
sampling a random arithmetic progression of length $q$ can be approximated
by sampling $q$ random weight tuples uniformly, under some additional
constraints.

In case of Theorem~\ref{thm:removal} these constraints have the form
of linear equations
with integer coefficients. For illustration, we show below a statement
for $q=3$, which is equivalent to the same-set case of Theorem~\ref{thm:removal}
(a more general statement that we need for full Theorem~\ref{thm:removal}
is given as Theorem~\ref{thm:feasible-removal}).
For $N > 0$, let $[-N, N]$ denote the set
$\left\{n \in \mathbb{Z}: |n| \le N\right\}$.

\begin{restatable}{theorem}{twodimensional}
  \label{thm:two-dimensional}
  Let $\mu > 0$ and let $A_1, B_1, A_2, B_2$ be i.i.d.~uniform
  in $[-N, N]$. There exists $\delta := \delta(\mu) > 0$
  such that
  for every subset $R \subseteq [-N, N]^2$ with
  density $\mu(R) := |R|/(2N+1)^2 \ge \mu$ we have
  \begin{align*}
    \Pr\left[
    (A_1, B_1) \in R \land (A_2, B_2) \in R \land
    (A_1+B_1-B_2, A_2+B_2-A_1) \in R
    \right] \ge \delta \; .
  \end{align*}
\end{restatable}

To prove Theorems~\ref{thm:feasible-removal} and~\ref{thm:two-dimensional} we use
a (hyper)graph removal lemma argument, similar as
in the classical proof of Szemerédi's theorem or in works on
removal lemmas for sets of linear equations
\cite{KSV09, Sha10, KSV12}. This application of graph removal makes
the constant $\delta$ to be very small compared to the density $\mu$
and we do not attempt to make it explicit.

Moving to Theorem~\ref{thm:poly}, it turns out that, since there are more
progressions to choose from, there is a larger collection of possible
arrangements of $p$ weight tuples. As a result, the constraints in the
relevant problem over $\mathbb{Z}^{p-1}$ turn out to be linear equations
modulo $p$, which can be directly handled in an easier and more abstract
fashion, at least for prime $p$.

While the restriction to symmetric sets is a strong one and the application
of central limit theorem might be considered quite natural, the author
finds it interesting that this technique results in an easier proof
and a better bound in Theorem~\ref{thm:poly} as compared to
Theorem~\ref{thm:removal}.

\subsection{Correlated spaces}

One can view the finite field problems we consider as instances in a more
general framework of \emph{correlated product spaces}.
Namely, let $\Omega$ be a finite set, $\ell \ge 2$
and $\cP$ a probability distribution
over $\Omega^{\ell}$ such that all of its $\ell$ marginals are uniform over
$\Omega$. We call such a distribution $\cP$ an
\emph{$\ell$-step correlated space}.
We consider the product probability space with $n$ i.i.d.~coordinates,
where coordinate $i \in [n]$ gives rise to a random tuple
$X_i^{(1)}, \ldots, X_i^{(\ell)}$ distributed according to $\cP$.

The random variables $X_i^{(j)}$ form $\ell$ random vectors
$X^{(j)} := \left(X_1^{(j)},\ldots,X_n^{(j)}\right)$.
Each of those vectors is individually
uniform in $\Omega^n$, but their joint distribution exhibits
correlation across the steps. We consider a setting
with fixed correlated space and $n$ going to infinity.

Most generally, given sets $S^{(1)}, \ldots, S^{(\ell)} \subseteq \Omega^n$
with densities $\mu^{(1)}, \ldots, \mu^{(\ell)}$ we want to study the
probability
\begin{align*}
  \Pr\left[X^{(1)}\in S^{(1)} \land \ldots \land X^{(\ell)} \in S^{(\ell)}\right]\;.
\end{align*}

For example, one can ask about the \emph{same-set} case
$S^{(1)} = \ldots = S^{(\ell)} = S$ with $\mu := \mu(S) > 0$.
That is, for a given correlated space we can ask if there exists a bound 
\begin{align}
  \label{eq:19}
  \Pr\left[ X^{(1)} \in S \land\ldots\land X^{(\ell)}\in S\right]
  \ge c\left(\cP, \mu\right) > 0 \;?
\end{align}
This problem was introduced in \cite{HHM18} and we call a space satisfying
\eqref{eq:19} \emph{same-set hitting}.
Note that indeed the capset problem is captured by
the same-set hitting on the three-step correlated space
where $\Omega = \mathbb{F}_3$ and $\cP$ is uniform in the set
of progressions of length three, i.e.,
$\left\{000,111,222,012,120,201,021,102,210\right\}$.

Considering ``dictator'' sets, for which the membership depends on a
single coordinate, it is easy to see that a necessary condition for
same-set hitting is that the diagonal
\begin{align*}
  \diag(\Omega) := \left\{ (\omega,\ldots,\omega): \omega \in\Omega\right\}
\end{align*}
is contained in the support of $\cP$. In \cite{HHM18} we proved that
this condition is sufficient for $\ell=2$. As a matter of fact, we state
the following conjecture:
\begin{conjecture}
  \label{con:same-set}
  Every correlated space with $\diag(\Omega) \subseteq \supp(\cP)$
  is same-set hitting.
\end{conjecture}

Generalizing Theorem~\ref{thm:removal} to arbitrary sets would
confirm Conjecture~\ref{con:same-set} in case of restricted arithmetic
progressions.
A related, more general question is if general removal lemma holds
for correlated product spaces:
\begin{question}
  \label{q:removal}
  Is it the case that for every correlated space $\cP$
  and every $\mu > 0$ there exists $\delta(\cP, \mu) > 0$ such that
  if
  \begin{align*}
    \Pr\left[X^{(1)} \in S^{(1)} \land\ldots\land X^{(\ell)}\in S^{(\ell)}\right]
    < \delta \;,
  \end{align*}
  then it is possible to remove a set $S'$ of density at most $\mu$
  from $S^{(1)}, \ldots, S^{(\ell)}$ and obtain $T^{(j)} := S^{(j)}\setminus S'$
  with
  \begin{align*}
    \Pr\left[X^{(1)} \in T^{(1)} \land\ldots\land X^{(\ell)}\in T^{(\ell)}\right]
    =0 \;?
  \end{align*}  
\end{question}

For all the author knows, we cannot even exclude a positive answer
to Question~\ref{q:removal} 
with $\delta > \mu^{C(\mathcal{P})}$ for every
correlated space $\mathcal{P}$. On the other hand, while it is
plausible that with some effort Theorems~\ref{thm:poly} and~\ref{thm:removal}
can be generalized to hold for arbitrary correlated spaces, in this work
we leave even this problem unresolved. We also leave open the problem of
characterizing the correlated spaces which allow a stronger polynomial bound
from Theorem~\ref{thm:poly}.
The class of spaces for which we can confirm
Conjecture~\ref{con:same-set} is limited and we discuss known results
in the following section.

\subsection{Related works}
We mention here some works that we find most relevant to our results and
proofs.

As we said before, one well-studied example of a correlated space
corresponds to the problem of arithmetic progressions in the finite field model.
Extensive recent line of work based on the
polynomial method
\cite{Gre05a, BX15, FK14, BCCGN, KSS18, Nor16, Peb18, FL17, FLS18, LS18}
culminated in establishing that for random \emph{$k$-cycles}, i.e.,
solutions to the equation $x_1+\cdots+x_k=0$ over finite field
$\mathbb{F}_p$ indeed the removal lemma holds
with $\delta > \mu^{C}$.

More generally, another interesting instance of a correlated space
arises when we take $\Omega=G$ for a group $G$ and $\cP$ is uniform over solutions
to some (full-rank) fixed linear equation system over $G$.
For example, a random arithmetic progression $a_1, \ldots, a_q$
over $\mathbb{Z}_q$
is a random solution of the equation system
$\left\{a_j + a_{j+2} = 2a_{j+1}\right\}_{j \in \{1,\ldots,q-2\}}$.
Green \cite{Gre05a} established such removal lemma for a single equation
and any abelian group $G$ (not
necessarily in the product setting)
and further work by Shapira~\cite{Sha10}
and Král', Serra and Vena~\cite{KSV12}
extended it to systems of equations over
finite fields, and it can be seen that their results carry over to the
product model $\mathbb{F}_p^n$.

Our proof of Theorem~\ref{thm:feasible-removal}, which is the second
part of the proof of Theorem~\ref{thm:removal}, is related to
this previous work on removal lemmas in systems of linear equations
in the following way: On the one hand,
the statement of Theorem~\ref{thm:feasible-removal}
is a removal lemma for a particular type of a system of linear equations.
Since it is a special system with some additional structure,
more involved constructions from \cite{Sha10} and \cite{KSV12}
are not required and we make a simpler
argument, similar as in the proof of Szemerédi's theorem or in \cite{KSV09}.
On the other hand, since we consider subsets $W \subseteq \mathbb{Z}^{q-1}$,
our result is not directly covered by \cite{Sha10} or \cite{KSV12},
which concern
only $W \subseteq \mathbb{Z}$.

Regarding Theorem~\ref{thm:removal}, the problem of restricted progressions
in $\mathbb{Z}_q^n$ seems to be particularly challenging, with most relevant
questions being wide open.
For example, it follows from known results that a \emph{linear subspace} of
$\mathbb{F}_3^n$ free of restricted progressions
has dimension at
most $n/2$ and that a subset of $\mathbb{F}_3^n$ that does not contain
a restricted progression \emph{of length two} must have size at most
$2^n$~\cite{Lev20}. It is also known~\cite{HHM18}
that every subset of $\mathbb{F}_3^n$ of
density $\mu$ contains at least $\delta(\mu)\cdot 6^n$ restricted progressions
of length two, where $\delta$ is a triply exponentially small in $\mu$.

More generally, a paper by Cook and Magyar \cite{CM12} shows that a set
of constant density $S \subseteq \mathbb{F}_p^n$ in the finite field model
contains a constant proportion of arithmetic progressions
with differences restricted to lie in a sufficiently well-behaved
algebraic set. However, the author does not see how to apply their result
in a very restricted setting of differences from $\{0,1\}^n$.

One reason we find the framework of correlated spaces interesting
is that it encompasses
some important problems from analysis of discrete functions, with
applications in computer science.
A canonical example of this setting are two steps $\ell=2$ over
binary alphabet $\Omega=\{0,1\}$
with $\cP(00) = \cP(11) = (1-p)/2$, $\cP(01) = \cP(10) = p/2$ for some
$p \in [0, 1]$. More generally, one can take any correlated space $\cP$
and add to it a small amount of uniform noise, e.g., taking
$\cP' := (1-\eps)\cdot\cP + \eps\cdot\mathcal{U}$, where $\mathcal{U}$ is the uniform
distribution over $\Omega^\ell$.

It turns out that the theory of reverse hypercontractivity \cite{MOS13}
can be used to show that in such setting (and, more generally, whenever
$\supp(\cP) = \Omega^\ell$), one gets a general \emph{set hitting}:
For any $S^{(1)}, \ldots, S^{(\ell)}$ with $\mu(S^{(j)}) \ge \mu$
it holds that
\begin{align*}
  \Pr\left[X^{(1)}\in S^{(1)}\land\ldots\land X^{(\ell)}\in S^{(\ell)}\right]
  \ge \mu^{C} \; .
\end{align*}

More generally, \cite{HHM18} established Conjecture~\ref{con:same-set} 
for $\ell = 2$, as well as whenever a certain \emph{correlation} value
is bounded with $\rho(\cP) < 1$.
The latter condition intuitively corresponds to the following:
For all possible assignments of values to $\ell-1$ of the steps in $\cP$, the value
of the remaining step is not determined. Note that this is quite a different regime
that what is usually encountered in additive combinatorics. For example,
the condition does not hold for full-rank systems of $r$ linear equations over
$m$ variables, where fixing $m-r$ variables determines the values of the remaining
$r$ variables.
\cite{HHM18} is based on the invariance principle
by Mossel \cite{Mos10}, which together with a follow-up
work\footnote{The technique from \cite{Mos17}
  implies another proof of same-set hitting for spaces
  with $\rho(\cP)<1$.} \cite{Mos17}
establishes set-hitting and, more precisely, Gaussian bounds, in
spaces with $\rho(\cP)<1$ for
sets with small low-degree Fourier coefficients.

A work by Friedgut and Regev \cite{FR18} applied the invariance principle
and previous work with Dinur \cite{DFR08} to establish a removal lemma
in the two-step case $\ell=2$. This removal lemma has tower-type dependence
between $\mu$ and $\delta$, which is worth contrasting with
\cite{HHM18} which established an easier property of same-set hitting
but with ``only'' triply exponential dependence between $\mu$ and $\delta$.
\cite{DFR08} and \cite{FR18} also studied the structure of sets
with hitting probability zero, establishing that any such set must be almost
contained in a junta.

The invariance principle can be compared with the Fourier-analytic approach
to Szemerédi's theorem due to Gowers \cite{Gow98, Gow01},
which takes as its starting point the fact that the space of arithmetic
progressions of length $k$ is set hitting for all sets with
low \emph{Gowers uniformity norm $U_k$}.

Finally, we note a work by Austrin and Mossel \cite{AM13} that
established set hitting for low-Fourier degree sets with small
Fourier coefficients in all correlated spaces where the distribution $\cP$
is pairwise independent.

\paragraph{Organization of the paper}
In the following we prove Theorems~\ref{thm:poly} and~\ref{thm:removal}.
In Section~\ref{sec:prelim} we introduce some notation, as well as the
local limit theorem we use in the remaining proofs.

For convenience of a less committed reader, in Section~\ref{sec:three}
we prove the same-set case of Theorem~\ref{thm:poly} for $q=3$.
This proof utilizes main ideas of our technique, while being somewhat
less technical and lighter in notation.

We proceed to prove Theorem~\ref{thm:poly} in Section~\ref{sec:poly}
and Theorem~\ref{thm:removal} in Section~\ref{sec:full}.
Each of the latter three sections is intended to be self-contained.

\paragraph{Acknowledgements}
I am grateful to Thomas Holenstein for suggesting the problem
and to Elchanan Mossel for helpful discussions and encouragement to write
this paper. I also thank the referee for helpful recommendations.

\section{Preliminaries}
\label{sec:prelim}

We use both $O(\cdot)$ and $\Omega(\cdot)$ asymptotic notation, as well as
constants $C > 0$ that will vary from time to time.
All such implicit constants are allowed to depend on the alphabet size denoted
by $p$ or $q$.

Given $x\in\mathbb{Z}_q^n$ and $a\in\mathbb{Z}_q$, we define the respective
\emph{weight} to be $w_a(x) := |\{i\in [n]: x_i = a\}|$. In the context of
arithmetic progressions $x^{(1)},\ldots,x^{(q)}$ of length $q$,
we will often speak of \emph{weight tuples}
$w^{(j)} = (w^{(j)}_0,\ldots,w^{(j)}_{q-2})$, where coordinates of $w^{(j)}$
will be weights $w_a(x^{(j)})$ shifted by a normalizing term approximately
equal to $n/q$. A collection of $q$ weight tuples
$w = (w^{(1)},\ldots,w^{(q)})$ will be referred to as a \emph{weight arrangement}.

In some of the estimates we employ standard notation
$\|x\|_2=\sqrt{\sum_{i=1}^n x_i^2}$ and
$\|x\|_\infty=\max_{i=1,\ldots,n}|x_i|$ for $x\in\mathbb{R}^n$.

We will apply several times the following corollary 
of the local multidimensional central limit theorem
(see, e.g., Chapter~5 in \cite{BR10} or Section~7 in \cite{Spi76}):
\begin{theorem}
  \label{cor:lclt}
  Let $W_1, \ldots, W_n$ be i.i.d.~random tuples such that each
  $W_i = (W^{(1)}_i, \ldots, W^{(\ell)}_i)$ is distributed uniformly
  in
  \begin{align*}
    \{0^\ell\} \cup \left\{ 0^j 1 0^{\ell-j-1}: 0 \le j \le \ell-1\right\}\;.
  \end{align*}

  For any tuple $w \in \mathbb{Z}^{\ell}$ with
  $d := w - \frac{n}{\ell+1}\cdot(1,\ldots,1)$
  we have
  \begin{align}
    \Pr\left[\sum_{i=1}^n W_i = w\right] &=
    \frac{(\ell+1)^{(\ell+1)/2}}{(2\pi)^{\ell/2}}
    \cdot \frac{1}{n^{\ell/2}} \cdot \exp\left(
    - \frac{\ell+1}{2n}\left(\|d\|_2^2 + \left(\sum_{j=1}^\ell d_j\right)^2
                                           \right)\right)
    \nonumber\\
    &\qquad+o\left(\frac{1}{n^{\ell/2}}\cdot\min\left(
    1, \frac{n}{\|d\|_2^2}
    \right)\right)\; ,\label{eq:26}
  \end{align}
  where the error term converges uniformly in $w$. In particular, we have
  \begin{align}\label{eq:23}
    \frac{1}{Cn^{\ell/2}}\cdot\left(
    \exp\left(-\frac{C\|d\|_2^2}{n}\right)+o(1)\right)
    \le\Pr\left[\sum_{i=1}^n W_i=w\right]
    \le\frac{C}{n^{\ell/2}}
  \end{align}
  for some $C > 0$ that depends only on $\ell$.
\end{theorem}

\section{Restricted Progressions of Length Three}
\label{sec:three}

In this section we prove the same-set case of Theorem~\ref{thm:removal}
for $q=3$.
For simplicity of exposition we additionally assume that
$n$ is divisible by six.
We first show that our result is implied by Theorem~\ref{thm:two-dimensional}
and then prove Theorem~\ref{thm:two-dimensional}
via the triangle removal lemma. Let us start with the statement
of the theorem.
\begin{theorem}
  \label{thm:symmetric-hitting}
  Let $X, Y, Z \in \mathbb{Z}_3^n$ be a random arithmetic progression
  with difference restricted to lie in $\{0,1\}^n$, where $n$ is a multiple of six.
  For every symmetric set $S \subseteq \mathbb{Z}_3^n$ with density
  $\mu(S) \ge \mu > 0$ we have
  \begin{align}
    \label{eq:05}
    \Pr \left[ X \in S \land Y \in S \land Z \in S \right] \ge \delta(\mu) > 0 \; .
  \end{align}
\end{theorem}

The crucial part of the proof is a lemma that
characterizes which weight arrangements are likely to be sampled
in a random restricted progression.
For this purpose, it is useful to introduce two random tuples.
The first one is
\begin{align*}
  M := (M_{000}, M_{111}, M_{012}, M_{120}, M_{201}) \; ,
\end{align*}
where $M_{abc} := \left|\{i \in [n]: (X_i, Y_i, Z_i) = (a,b,c)\}\right| - n/6$
for $(a,b,c) = (x,x+d,x+2d), x\in\mathbb{Z}_3,d\in\{0,1\}$.
That is, the tuple $M$ expresses normalized counts of six restricted
progressions across $n$ coordinates.
Note that $M_{222}$ is omitted, since it can be inferred from the remaining
components of $M$.

The second random tuple represents weight arrangements of elements
of the restricted progression:
\begin{align*}
  W := (W_{X}, W_{Y}, W_{Z}) := (w_0(X), w_1(X), w_0(Y), w_1(Y), w_0(Z), w_1(Z)) -
  (n/3,\ldots,n/3) \; .
\end{align*}
Again, we omit weights $w_2(\cdot)$, since they can be deduced from the
rest. Note that $W$ is determined by $M$, but, as it turns out,
not the other way around.

\begin{lemma}
  \label{lem:admissible-probs}
  Let $(X, Y, Z)$ be a random restricted progression
  with $n$ divisible by six.
  Let
  $w := \allowbreak(x_0, x_1, y_0,\allowbreak y_1, z_0, z_1)
  \allowbreak \in [-N, N]^6$ with
  $N = C_1\sqrt{n}$ for some $C_1 > 0$. Then,
  \begin{align*}
    \Pr[W=w]\text{ is } \begin{cases}
      \text{at least } C_2/N^4 &\text{ if } (z_0, z_1) = (x_0+x_1-y_1, y_0+y_1-x_0) \; ,\\
      0 &\text{ otherwise,}
      \end{cases}
  \end{align*}
  for some $C_2 := C_2(C_1) > 0$ and $N$ large enough (also depending on $C_1$).
\end{lemma}

\begin{proof}
  Let us call a tuple $w$ that satisfies
  $(z_0, z_1) = (x_0 + x_1 - y_1, y_0 + y_1 - x_0)$ \emph{feasible}.
  If $w$ is not feasible, then clearly $\Pr[W=w]=0$, since restricting
  the progression difference to $\{0,1\}^n$ implies that
  \begin{align*}
    w_0(Z) &= M_{000}+M_{120}+n/3=w_0(X)+w_1(X)-w_1(Y)\;,\\
    w_1(Z) &= M_{111}+M_{201}+n/3=w_0(Y)+w_1(Y)-w_0(X)\;.
  \end{align*}
  For a feasible $w$, one can see that a tuple $m$ gives rise to
  the weight arrangement $w$ if and only if it is an integer solution of
  the equation system
  \begin{align*}\begin{cases}
      x_0 = m_{000}+m_{012}\\
      x_1 = m_{111}+m_{120}\\
      y_0 = m_{000}+m_{201}\\
      y_1 = m_{111}+m_{012}
    \end{cases}\;,\end{align*}
  and these solutions are given as
  \begin{align*}
    m = (m_{000}, m_{111}, m_{012}, m_{120}, m_{201}) =
    (k, y_1-x_0+k, x_0-k, x_1-y_1+x_0-k, y_0-k)
  \end{align*}
  for $k \in \mathbb{Z}$. Now we can calculate
  \begin{align}
    \Pr[W = w]
    &= \sum_{k \in \mathbb{Z}} \Pr\left[
      M=(k,y_1-x_0+k, x_0-k, x_1-y_1+x_0-k, y_0-k) \right]
    \nonumber\\
    &\ge \sum_{k = 0}^{\lfloor C_1\sqrt{n}\rfloor} \Pr\left[
      M=(k,y_1-x_0+k, x_0-k, x_1-y_1+x_0-k, y_0-k) \right] \;.
      \label{eq:04}
  \end{align}
  Each tuple $m = m(k)$ in the summation \eqref{eq:04}
  is contained in $[-4N, 4N]^5$ and therefore satisfies
  $\|m\|_2^2 = O(C_1^2\cdot n)$.
  Applying lower bound in \eqref{eq:23} to $M$ and $m$,
  each term in the summation \eqref{eq:04} must be
  at least $C/n^{5/2}$, where $C$ depends on $C_1$ and $n$ is large enough.
  Finally, summing up over $k$ we get
  \begin{align*}
    \Pr[W = w] \ge \frac{C}{n^2} = \frac{C_2}{N^4} \; ,
  \end{align*}
  as claimed.
\end{proof}

\begin{proof}[Theorem~\ref{thm:two-dimensional} implies
  Theorem~\ref{thm:symmetric-hitting}]
  Let $\mu > 0$ and $S \subseteq \mathbb{Z}_3^n$ be a symmetric set
  with $\mu(S) \ge \mu > 0$. Note that we can assume that $n$ is big enough.
  
  Recall the random tuple $W = (W_X, W_Y, W_Z)$ and for $x \in \mathbb{Z}_3^n$ let
  $W_x := (w_0(x)-n/3, w_1(x)-n/3)$.
  Since $S$ is symmetric, there exists a set $R := R(S) \subseteq \mathbb{Z}^2$
  such that $x \in S$ if and only if $W_x \in R$.
  Since, by a standard concentration bound, we can find $C(\mu)$ such that
  \begin{align*}
    \Pr\big[\left\|W_{X}\right\|_{\infty} > C\sqrt{n} \big]
    < \mu/2 \; ,
  \end{align*}
  from now on we will assume w.l.o.g.~that $R \subseteq [-N, N]^2$
  for $N := C\sqrt{n}$.

  Observe that, due to upper bound in \eqref{eq:23} applied to random
  variable $W_X$, for each $w \in R$ we have
  \begin{align*}
    \Pr\left[W_{X} = w \right] \le O(1/n)\; ,
  \end{align*}
  and therefore $|R| = \Omega(\mu/n)$, implying
  $\frac{|R|}{(2N+1)^2} \ge c(\mu) > 0$. Now, Theorem~\ref{thm:two-dimensional}
  gives
  \begin{align*}
    \frac{\#\left\{
    (x_0, x_1, y_0, y_1): (x_0, x_1) \in R \land (y_0, y_1) \in R
    \land (x_0+x_1-y_1, y_0+y_1-x_0) \in R
    \right\}}
    {(2N+1)^4}\\
    \ge c(\mu) > 0 \; ,
  \end{align*}
  but that, due to Lemma~\ref{lem:admissible-probs}, yields
  \begin{IEEEeqnarray*}{l}
    \Pr[X \in S \land Y \in S \land Z \in S]
    =\Pr[W_{X} \in R \land W_{Y} \in R \land W_{Z} \in R]\\
    \qquad=
    \sum_{\substack{(x_0,x_1)\in R\\(y_0,y_1)\in R}}
    \mathbbm{1}_R(x_0+x_1-y_1,y_0+y_1-x_0)
    \Pr\left[W = (x_0,x_1,y_0,y_1,x_0+x_1-y_1,y_0+y_1-x_0)\right]\\
    \qquad\ge c(\mu) \cdot N^4 \cdot \frac{C_2(\mu)}{N^4} \ge \delta(\mu) > 0 \; .
  \end{IEEEeqnarray*}
\end{proof}

We proceed to the proof of Theorem~\ref{thm:two-dimensional}, which we
restate here for convenience.
\twodimensional*

The proof is a variation on
the triangle removal proof of Roth's theorem. Let us start by stating the
removal lemma:
\begin{theorem}[Triangle removal lemma]
  For every $\eps > 0$ there exists $\delta = \delta(\eps) > 0$
  such that if a simple graph $G = (V, E)$ contains at most $\delta |V|^3$
  triangles, then it is possible to make $G$ triangle-free by removing
  from it at most
  $\eps|V|^2$ edges.
\end{theorem}

\begin{proof}[Proof of Theorem~\ref{thm:two-dimensional}]
  Let $N \in \bbN$ and $R \subseteq [-N, N]^2$ with density
  $\mu(R) \ge \mu > 0$.
  As before, we will call a triple of points
  $(a_1, b_1), (a_2, b_2), (a_3, b_3) \in \mathbb{Z}^2$
  \emph{feasible} if $(a_3, b_3) = (a_1+b_1-b_2, a_2+b_2-a_1)$.
  We define a tripartite graph $G$ as follows:
  \begin{itemize}
  \item There are three groups of vertices $V_1, V_2, V_3$. In each group
    the vertices are labeled with elements of $[-M, M]^2$ for $M := 3N$.
    Note that the total number of vertices of $G$ is $|V| = 3(2M+1)^2$.
  \item Edge adjacency is defined by:
    \begin{align}
      V_1 \ni (i_a, i_b) \sim (j_a, j_b) \in V_2
      &\text{ iff }(a_1, b_1) := (i_b-j_b, j_a-i_a+j_b-i_b) \in R \; ,
      \label{eq:01}\\
      V_2 \ni (j_a, j_b) \sim (k_a, k_b) \in V_3
      &\text{ iff }(a_2, b_2) := (k_a-j_a+k_b-j_b, j_a-k_a) \in R \; ,
      \label{eq:02}\\
      V_1 \ni (i_a, i_b) \sim (k_a, k_b) \in V_3
      &\text{ iff }(a_3, b_3) := (k_a-i_a, k_b-i_b) \in R \; .
        \label{eq:03}
    \end{align}
  \end{itemize}
  Given a triple of vertices $(i_a, i_b), (j_a, j_b), (k_a, k_b)$, we
  associate with it a triple of points
  $(a_1, b_1), (a_2, b_2),\allowbreak (a_3, b_3) \in \mathbb{Z}^2$
  given by the right-hand sides of equations in \eqref{eq:01} to \eqref{eq:03}.
  One checks that this triple of points is feasible. Furthermore, by
  definition,
  whenever $(i_a, i_b), (j_a, j_b), (k_a, k_b)$ form a triangle,
  the points $(a_1, b_1), (a_2, b_2),\allowbreak (a_3, b_3)$ all belong to $R$.

  Conversely, given a point $(a, b) \in R$, we can see that each triple
  of vertices $(i_a, i_b), (i_a+a+b, i_b-a), (i_a+a, i_b+b)$
  for $(i_a, i_b) \in [-N, N]^2$ forms a triangle.
  Therefore, the graph $G$ contains at least
  $\mu\cdot (2N+1)^4 \ge \mu\cdot \left(\frac{2M+1}{3}\right)^4 = \frac{\mu}{3^6}|V|^2$
  triangles. Furthermore, it is clear that all those triangles are edge-disjoint.
  Hence, $G$
  requires at least $\frac{\mu}{3^6}|V|^2$
  edge deletions to become triangle-free and, by triangle removal lemma,
  contains at least $\delta(\mu)|V|^3$ triangles.

  Finally, we note that each feasible triple of points
  $(a_1, b_1), (a_2, b_2), (a_3, b_3) \in R$
  gives rise to at most $(2M+1)^2$ triangles. This is because each vertex
  $(i_a, i_b) \in V_1$ determines at most one triangle associated with this triple.
  Since $G$ contains at least $\delta|V|^3$ triangles, the number of
  feasible triples contained in $R$ must be at least
  \begin{align*}
    \frac{\delta|V|^3}{(2M+1)^2} = 27\delta (2M+1)^4 \ge \delta(2N+1)^4 \; ,
  \end{align*}
  but this means
  \begin{align*}
    \Pr\left[
    (A_1, B_1) \in R \land (A_2, B_2) \in R \land
    (A_1+B_1-B_2, A_2+B_2-A_1) \in R
    \right] \ge \delta > 0 \; ,
  \end{align*}
  as we wanted.
\end{proof}

\section{Proof of Theorem~\ref{thm:poly}}
\label{sec:poly}

In this section we prove Theorem~\ref{thm:poly}.

As a very preliminary point, note that it suffices to consider only densities
$\mu \le 1-1/p$ that are bounded away from one.
If $\mu$ is greater than $1-1/p$, one can, for example,
take $T^{(j)} := S^{(j)} \cap \{x: \sum_{i=1}^n x_i = 0\pmod{p} \}$
for $j\ne p$ and $T^{(p)} := S^{(p)} \cap \{x:\sum_{i=1}^nx_i=1\pmod{p}\}$
in order to obtain progression-free sets $T^{(1)},\ldots,T^{(p)}$.

As in the proof of Theorem~\ref{thm:symmetric-hitting}, we start with observing
that there exist sets $R^{(1)},\ldots,R^{(p)} \subseteq \mathbb{F}^{p-1}$
such that $x\in S^{(j)}$ if and only if
$W(x) := (w_0(x)-p\lfloor n/p^2 \rfloor,\ldots,w_{p-2}(x)-
p\lfloor n/p^2\rfloor)\in R^{(j)}$
(the choice of the $p\lfloor n/p^2\rfloor$ shift will become apparent
in the next paragraph).
Furthermore, using standard concentration
bound
\begin{align*}
  \Pr\left[ \left|w_a(X)-n/p\right| > tn \right] \le 2\exp(-2nt^2)
\end{align*}
for any fixed $a\in\mathbb{F}_p$ together with the union bound,
we also establish that there exists some $C_1 > 0$ such that
\begin{align*}
  \Pr\left[ \left\|W(X^{(j)})\right\|_{\infty} >
  C_1\sqrt{n\cdot\ln 1/\mu} \right] \le \frac{\mu}{2} \; ,
\end{align*}
and therefore we can remove from each of $S^{(1)},\ldots,S^{(p)}$
a symmetric set of density at most $\mu/2$ and assume from now on
that the weight sets $R^{(1)},\ldots,R^{(p)}$ have limited range:
If $w\in R^{(j)}$, then $\|w\|_{\infty} \le C_1\sqrt{n\cdot\ln 1/\mu}$.

For $a,d\in \mathbb{F}_p$, define random variables
\begin{align*}
  M(a,d) := \left|\left\{i\in [n]: x^{(1)}_i = a \land x^{(2)}_i-x^{(1)}_i = d
  \right\}\right| - \lfloor n/p^2 \rfloor \; .
\end{align*}
Consider the random tuple $M$ consisting of $p^2-1$ coordinates $M(a,d)$
except for $M(p-1,0)$.
We note for future reference that $M$ can be written as a sum
of i.i.d.~random tuples $M=\sum_{i=1}^n M_i$ such that
Theorem~\ref{cor:lclt} is applicable.
We also note that there exists a matrix
$A \in \{0,1\}^{p(p-1)\times (p^2-1)}$ such that letting
$W := (W(X^{(1)}),\ldots,W(X^{(p)}))$ we can write a linear system of equations
$W = AM$.

At this point we need to understand how solutions to the system $W=AM$ look
like. This is done in the following lemma:
\begin{lemma}\label{lem:linear-structure}
  A general solution to the equation system $w=Am$ for $w\in\mathbb{R}^{p(p-1)}$
  is given by
  \begin{align}\label{eq:24}
    m = \frac{1}{p}Bw + K\cdot \mathbb{R}^{p-1}\;
  \end{align}
  for some \emph{integer-valued} matrices $B\in\mathbb{Z}^{(p^2-1)\times p(p-1)}$
  and $K\in\mathbb{Z}^{(p^2-1)\times (p-1)}$. In particular, matrix $A$ has full
  rank over reals.

  Furthermore, if  the vector $\frac{1}{p}Bw$
  is not integer for some integer $w\in\mathbb{Z}^{p(p-1)}$,
  then the system $w=Am$ does not have an
  integer solution.
\end{lemma}

\begin{proof}
  We start by showing that matrix $A$ has full rank with a solution given
  as $m = \frac{1}{p}B'w$ for some $B'\in\mathbb{Z}^{(p^2-1)\times p(p-1)}$.
  To do this, we explicitly construct the columns of $B'$. That is, for
  every $j\in[p]$ and $a \in \mathbb{F}_p, a\ne p-1$,
  we find a solution $m_{j,a}\in\mathbb{R}^{p^2-1}$ to equation
  $w_{j,a} = Am_{j,a}$  where $w_{j,a}$ has value one at the coordinate
  corresponding to $W_a(x^{(j)})$ and zero everywhere else. Furthermore,
  this solution satisfies $p\cdot m_{j,a} \in \mathbb{Z}^{p^2-1}$.
  We give $m_{j,a}$ as:
\begin{align}\label{eq:25}
  m_{j,a}(b,d) := \begin{cases}
    -(p-2)/p&\text{if }b=a,d=0\;,\\
    -(p-1)/p&\text{if }b\ne a,d=0\;,\\
    2/p&\text{if }d\ne 0,b+(j-1)d=a\;,\\
    1/p&\text{if }d\ne 0,b+(j-1)d\notin\{a,p-1\}\;,\\
    0&\text{if }d\ne 0,b+(j-1)d=p-1\;.
    \end{cases}
\end{align}
As a sanity check we can convince ourselves that $m_{j,a}$ features one
coordinate with value $-(p-2)/p$, $p-2$ coordinates with value
$-(p-1)/p$, $p-1$ coordinates with value $2/p$,
$(p-1)(p-2)$ coordinates with value $1/p$ and $p-1$ coordinates with value zero.
Indeed, we have $w_{j,a} = Am_{j,a}$ as can be seen by indexing
coordinates of $w=Am_{j,a}$ as $w(j', b)$ with
$j'\in[p],b\in\mathbb{F}_p\setminus\{p-1\}$ and computing
\begin{IEEEeqnarray*}{rClr}
  w(j, a)&=&-\frac{p-2}{p} + (p-1)\cdot\frac{2}{p}=1&\\
  w(j, b)&=&-\frac{p-1}{p}+(p-1)\cdot\frac{1}{p}=0
  &\qquad\text{for }b\notin\{a,p-1\}\;,\\
  w(j', a)&=&-\frac{p-2}{p}+(p-2)\cdot\frac{1}{p}=0
  &\qquad\text{for }j'\ne j\;,\\
  w(j', b)&=&-\frac{p-1}{p}+\frac{2}{p}+(p-3)\cdot\frac{1}{p}=0
  &\qquad\text{for }b\notin\{a,p-1\},j'\ne j\;.
\end{IEEEeqnarray*}
By a similar check
we can characterize the $(p-1)$-dimensional kernel of the linear operator
$A$ concluding that $Am = 0$ holds if
\begin{align}\label{eq:20}
  m(b,d) := \begin{cases}
    -\sum_{d'=1}^{p-1}\alpha_{d'}&\text{if }d=0\;,\\
    \alpha_d&\text{if }d\ne 0\;,\\
    \end{cases}
\end{align}
for $\alpha_1,\ldots,\alpha_{p-1}\in\mathbb{R}$. Since matrix $A$ is full rank,
its kernel has dimension $p-1$ and equation \eqref{eq:20} represents
all elements in the kernel.

Combining~\eqref{eq:25} and~\eqref{eq:20} allows us to write
a general solution to $w = Am$ as
\begin{align*}
  m = \frac{1}{p}B'w + K\cdot\mathbb{R}^{p-1}
\end{align*}
for $B'\in\mathbb{Z}^{(p^2-1)\times p(p-1)}$ and $K\in\mathbb{Z}^{(p^2-1)\times (p-1)}$.

As for the ``furthermore'' claim, notice that another general solution
to $w = Am$ can be obtained by adding an arbitrary kernel vector
$Kv$  to one of the columns of $B'$. Applying this observation repeatedly
together with~\eqref{eq:20}, we obtain another integer matrix $B$ such that
we still have the
equation
\begin{align*}
  m = \frac{1}{p}Bw + K\cdot\mathbb{R}^{p-1}
\end{align*}
and, additionally, for every column $b_{j,a}$ of $B$
we have $b_{j,a}(0, d) = 0$ for $d=1,\ldots,p-1$. But this ensures that
a solution $m = \frac{1}{p}Bw + Kv$ has $m(0, d) = v_d$, so
$v$ must be integer in order for $m$ to be integer,
which implies that $\frac{1}{p}Bw$ is integer as well.
\end{proof}

Consider an integer tuple $w\in\mathbb{Z}^{p(p-1)}$ and a tuple
$w\pmod{p} \in\mathbb{F}_p^{p(p-1)}$ which consists of entries of $w$ reduced
modulo $p$.
Since by Lemma~\ref{lem:linear-structure} $w = Am$ has an integer
solution $m\in\mathbb{Z}^{p^2-1}$ if and only if $\frac{1}{p}Bw$ is
integer, in particular this property depends only
on $w\pmod{p}$.

Furthermore, if $w$ with
$\|w\|_{\infty}\le D$ has at least one integer
solution, then taking
$m = \frac{1}{p}Bw + K\cdot[D]^{p-1}$ we get
$D^{p-1}$ integer solutions $m$ with bounded norm $\|m\|_{\infty}\le O(D)$ and,
consequently, $\|m\|_2^2 \le O(D^2)$. Taking
$D = C_1\sqrt{n\ln 1/\mu}$ and
applying lower bound in~\eqref{eq:23} to random tuple $M$,
we see that, for $n$ big enough, we will have for each such $m$
\begin{align*}
  \Pr[M = m] \ge \frac{\mu^C}{n^{(p^2-1)/2}} \;,
\end{align*}
and, as a result,
\begin{align}\label{eq:21}
  \Pr[W = w] \ge \frac{\mu^C}{n^{(p^2-1)/2}}\cdot
  \left(C_1\sqrt{n\ln 1/\mu}\right)^{p-1}\ge
  \frac{\mu^C}{n^{p(p-1)/2}} \;.
\end{align}

To finish the proof, divide each set $S^{(j)}$ into $p^{p-1}$
``congruence classes'' $S_v^{(j)}$ indexed by
$v = (v_0,\ldots,v_{p-2}) \in \mathbb{F}_p^{p-1}$
and defined as
\begin{align*}
  S^{(j)}_v := S^{(j)}\cap
  \left\{x: w_a(x)-p\lfloor n/p^2\rfloor=v_a\pmod{p},\quad a=0,\ldots,p-2
  \right\}\;.
\end{align*}
Now, if $|S_v^{(j)}| \le \frac{\mu}{2p^{p-1}}\cdot p^n$, we remove
$S_v^{(j)}$ from $S^{(j)}$. Clearly, we removed from each $S^{(j)}$ a
symmetric set of density at most $\mu$. The final claim is that
there can be no tuple arrangement $v^{(1)},\ldots,v^{(p)}$ such that:
\begin{enumerate}
\item None of $S^{(j)}_{v^{(j)}}$ has been removed.
\item There exists an integer solution to $v = Am$, where
  $v = (v^{(1)},\ldots,v^{(p)})$.
\end{enumerate}
Otherwise, each of the sets $S^{(j)}_{v^{(j)}}$ has density at least
$\mu/2p^{p-1}$. A set $S^{(j)}_{v^{(j)}}$ is symmetric with the corresponding
set of tuples $R^{(j)}_{v^{(j)}}\subseteq \mathbb{Z}^{p-1}$ bounded by
$\|w^{(j)}\|_{\infty}\le C_1\sqrt{n \ln 1/\mu}$ for $w^{(j)}\in R^{(j)}_{v^{(j)}}$.
Applying Theorem~\ref{cor:lclt} to the random tuple
$W^{(j)} = (W_0^{(j)},\ldots,W_{p-2}^{(j)})$, we see that each such tuple
$w^{(j)}\in R^{(j)}_{v^{(j)}}$ has
\begin{align*}
  \Pr[W^{(j)} = w^{(j)}] \le O\left(\frac{1}{n^{(p-1)/2}}\right)\;,
\end{align*}
and therefore, we can bound the size of $R^{(j)}_{v^{(j)}}$ by
\begin{align*}
  \left|R^{(j)}_{v^{(j)}}\right|\ge\Omega\left(
  \mu\cdot n^{(p-1)/2}
  \right)\;.
\end{align*}
As a result, we get a set of $\Omega\left(\mu^{p}\cdot n^{p(p-1)/2}\right)$
weight arrangements $w = (w^{(1)},\ldots,w^{(p)})$ with
$\|w\|_{\infty} \le C_1\sqrt{n\ln 1/\mu}$
and for each of them
the system $w = Am$ has an integer-valued solution and, by~\eqref{eq:21},
$\Pr[W=w]\ge\frac{\mu^C}{n^{p(p-1)/2}}$, which finally gives us
\begin{align*}
  \Pr\left[X^{(1)}\in S^{(1)}\land\ldots\land X^{(p)}\in S^{(p)}\right] \ge
  \frac{1}{C}\cdot\mu^p\cdot n^{p(p-1)/2}\cdot\frac{\mu^C}{n^{p(p-1)/2}}
  \ge
  \mu^C\;,
\end{align*}
which contradicts assumption~\eqref{eq:22}  if the constant $C$ is chosen large
enough.

We established that there are no ``mod $p$'' weight arrangements
$(v^{(1)},\ldots,v^{(p)})$ that satisfy the two conditions above.
But it follows that there are
no weight arrangements $w = (w^{(1)},\ldots,w^{(p)})$
left in $R^{(1)}\times\cdots\times R^{(p)}$
for which there is an integer
solution to $w = Am$, and therefore no arithmetic progressions left
in the product set $S^{(1)}\times\cdots\times S^{(p)}$, and we are done.
\qed

\begin{remark}
  We make no attempt to precisely estimate the constant $C$ in the exponent,
  but following the argument above one can see that it is bounded by a
  polynomial function of $p$.
\end{remark}

\section{Proof of Theorem~\ref{thm:removal}}
\label{sec:full}

In the following we prove Theorem~\ref{thm:removal}.
We start with some definitions:
\begin{definition}
  For $q \ge 3$ and $n \ge 1$, we let
  \begin{align*}
    \mathcal{P} := \mathcal{P}(q, n) := \left\{
    (x, x+d, \ldots, x+(q-1)d):
    x \in \mathbb{Z}_q^n, d \in \left\{0,1\right\}^n
    \right\}
  \end{align*}
\end{definition}
Note that $\left|\cP(q, n)\right| = (2q)^n$. We will call an element
of $\cP(q, n)$ a \emph{restricted progression}.
We will restate our removal lemma now:
\begin{theorem}
  \label{thm:product-removal}
  For all $\mu > 0$ there exists
  $\delta := \delta(q, \mu) > 0$ such that for all symmetric sets
  $S^{(1)}, \ldots, S^{(q)} \subseteq \mathbb{Z}_q^n$:

  If $S^{(1)} \times \ldots \times S^{(q)}$ contains at most $\delta\cdot(2q)^n$
  restricted progressions, then it is possible to remove a total number
  of at most $\mu q^n$ elements from $S^{(1)}, \ldots, S^{(q)}$ and obtain
  symmetric sets $T^{(1)}, \ldots, T^{(q)}$ such that
  $T^{(1)} \times \ldots \times T^{(q)}$ contains \emph{no}
  restricted progressions.
\end{theorem}

As before, the proof consists of two parts: First, we make a CLT argument
reducing Theorem~\ref{thm:product-removal} to a variation on removal lemma for
certain linear equations over $\mathbb{Z}_N^{q-1}$. Then, we apply the hypergraph
removal lemma to establish the linear equation removal property.
To state the removal property over $\mathbb{Z}_N^{q-1}$ we need another
definition specifying the allowed weight arrangements of
restricted progressions.
\begin{definition}
  For $x \in \mathbb{Z}_q^n$, we define the \emph{weight tuple} of $x$ as
  $W(x) := (W_1(x), \ldots, W_{q-1}(x)) \in \mathbb{Z}^{q-1}$, where
  \begin{align*}
    W_{a}(x) := w_{a}(x)-2\lfloor n/2q\rfloor = \left|\left\{
    i \in [n] : x_i = a
    \right\}\right| - 2\lfloor n/2q \rfloor \; .
  \end{align*}
\end{definition}

\begin{definition}
  \label{def:feasible}
  An arrangement of tuples
  $(w^{(1)}_1, \ldots, w_{q-1}^{(1)}), \ldots,
  (w_1^{(q)}, \ldots, w_{q-1}^{(q)})$
  (understood, depending on the context,
  as element of $\mathbb{Z}^{q(q-1)}$ or $\mathbb{Z}_N^{q(q-1)}$) is \emph{feasible}
  if
  \begin{align}
    &\forall j = 3, \ldots, q:\nonumber\\
    &\qquad w_1^{(j)} = \sum_{a=1}^{q-1} w_a^{(j-2)} -
      \sum_{a=2}^{q-1} w_a^{(j-1)}\label{eq:14}\\
    &\qquad \forall a = 2, \ldots, q-1:
      w_a^{(j)} = -w_{a-1}^{(j-2)} + w_{a-1}^{(j-1)} + w_a^{(j-1)}\label{eq:15}
  \end{align}
  For $q \ge 3$ and $N \ge 1$ we let
  $\cE := \cE(q, N) \subseteq \mathbb{Z}_N^{q(q-1)}$ to be the set of all
  feasible arrangements of tuples. 
\end{definition}
Note that $|\cE(q, N)| = N^{2(q-1)}$.
The definition of a feasible tuple is motivated by the following claim,
which can be seen to be true by inspection:
\begin{claim}
  \label{cl:progression-feasible}
  Let $x^{(1)}, \ldots, x^{(q)} \in \mathbb{Z}_q^n$ be a restricted progression.
  Then, the weight arrangement $w(x^{(1)}), \ldots,\allowbreak w(x^{(q)}) \in \mathbb{Z}^{q(q-1)}$
  is feasible.
\end{claim}

Finally, we are ready to state the removal property for feasible arrangements.
In this case it seems slightly more convenient (but not much different)
to work in the cyclic group $\mathbb{Z}_N$ rather than in $\mathbb{Z}$.
\begin{theorem}
  \label{thm:feasible-removal}
  For all $\mu > 0$ there exists $\delta := \delta(q, \mu) > 0$ such that
  for all sets $R^{(1)}, \ldots, R^{(q)} \subseteq \mathbb{Z}_N^{q-1}$:
  
  If the product $R^{(1)} \times \ldots \times R^{(q)}$ contains at most
  $\delta N^{2(q-1)}$ feasible arrangements, then it is possible to remove
  a total number of at most $\mu N^{q-1}$ tuples from $R^{(1)}, \ldots, R^{(q)}$
  and obtain sets $R'^{(1)}, \ldots, R'^{(q)}$ such that the product
  $R'^{(1)} \times \ldots \times R'^{(q)}$ contains no feasible tuple
  arrangements.
\end{theorem}

\subsection{Theorem~\ref{thm:feasible-removal} implies
  Theorem~\ref{thm:product-removal}}

The CLT argument that we use to prove that Theorem~\ref{thm:product-removal}
is implied by Theorem~\ref{thm:feasible-removal} can be encapsulated in the
following lemma that will be proved last. Before stating the lemma
we need one more definition:
\begin{definition}
  Let $w \in \mathbb{Z}^{q-1}$ be a weight tuple and
  $w^{(1)}, \ldots, w^{(q)} \in \mathbb{Z}^{q(q-1)}$ a weight arrangement.
  We let
  \begin{align*}
    \# w
    &:= \left|\left\{
      x \in \mathbb{Z}_q^n: w(x) = w 
      \right\}\right|\\
    \# \left(w^{(1)}, \ldots, w^{(q)}\right)
    &:= \left|\left\{
      \left(x^{(1)}, \ldots x^{(q)}\right) \in \cP:
      \ \forall j = 1, \ldots, q:
      w(x^{(j)}) = w^{(j)}
      \right\}\right|
  \end{align*}
\end{definition}
\begin{lemma}
  \label{lem:clt-counts}
  Let $q \ge 3$, $C_1 > 0$ and let $N := C_1 \sqrt{n}$.
  For any weight tuple $w \in [-N, N]^{q-1}$ and any weight arrangement
  $w^{(1)}, \ldots, w^{(q)} \in [-N, N]^{q(q-1)}$ we have the following:
  \begin{enumerate}
  \item If $w^{(1)}, \ldots, w^{(q)}$ is not feasible, then
    $\#\left(w^{(1)}, \ldots, w^{(q)}\right) = 0$.
  \item If $w^{(1)}, \ldots, w^{(q)}$ is feasible, then
    \begin{align}
      \frac{1}{C}\le
      \#\left(w^{(1)}, \ldots, w^{(q)}\right)\cdot\frac{N^{2(q-1)}}{(2q)^n}
      \le C\;,
      \label{eq:11}\end{align}
    for large enough $n$ and some $C > 0$ that may depend on $C_1$.
  \item Similarly,
    $\frac{1}{C}\le\# w\cdot\frac{N^{q-1}}{q^n}\le C$ for large enough $n$.
  \end{enumerate}
\end{lemma}

\begin{proof}[Proof of Theorem~\ref{thm:product-removal} assuming
  Theorem~\ref{thm:feasible-removal} and Lemma~\ref{lem:clt-counts}]
  Let $S^{(1)}, \ldots, S^{(q)} \subseteq \mathbb{Z}_q^n$ be the sets
  from the statement. Since they are symmetric, there are sets
  $R^{(1)}, \ldots, R^{(q)} \subseteq \mathbb{Z}^{q-1}$ such hat
  \begin{align*}
    x \in S^{(j)} \iff W(x) \in R^{(j)} \; .
  \end{align*}
  
  The first observation is that we can assume without loss of generality
  that $n$ is large and that the weights are restricted such that
  $R^{(j)} \in [-N, N]^{q-1}$ for $N := C_1\sqrt{n}$
  for some $C_1 := C_1(q, \mu) > 0$.  
  This is because by a standard concentration bound
  \begin{align*}
    \left| \left\{ x: W(x) \notin [-N, N]^{q-1} \right\}
    \right| \le \frac{\mu}{2q}q^n
  \end{align*}
  for a big enough $C_1$ and therefore it takes at most $\mu/2 \cdot q^n$
  removals to get rid of all the elements giving rise to weight
  tuples outside $[-N, N]^{q-1}$.

  By Lemma~\ref{lem:clt-counts}.2, there exists some $C:=C(q,C_1) > 0$
  such that 
  each feasible arrangement in
  $R^{(1)} \times \ldots \times R^{(q)}$ induces at least
  $\frac{(2q)^n}{C\cdot N^{2(q-1)}}$ restricted progressions in
  $S^{(1)} \times \ldots \times S^{(q)}$.
  Let $\mu' := \frac{\mu}{2C(2q)^{q-1}}$ and let $\delta'(\mu') > 0$ be given by
  Theorem~\ref{thm:feasible-removal}. We set
  $\delta(\mu) := (2q)^{2(q-1)}\delta'/C$.

  Since, by assumption,
  $S^{(1)} \times \ldots \times S^{(q)}$ contains at most $\delta(2q)^n$ restricted
  progressions,
  $R^{(1)} \times \ldots \times R^{(q)}$ contains
  at most $\delta CN^{2(q-1)}$
  feasible arrangements (understood as elements of $\mathbb{Z}^{q(q-1)}$).
  Furthermore, taking $N' := 2qN$ and inspecting Definition~\ref{def:feasible},
  we conclude that $R^{(1)} \times \ldots \times R^{(q)}$ contains at most
  $\frac{\delta C}{(2q)^{2(q-1)}} N'^{2(q-1)} = \delta' N'^{2(q-1)}$
  feasible arrangements understood as elements of
  $\mathbb{Z}_{N'}^{q(q-1)}$.

  Applying Theorem~\ref{thm:feasible-removal} for $N'$ and
  $\mu'$, we get that one can remove at most
  $\mu' N'^{q-1} = \frac{\mu}{2C} N^{q-1}$ elements
  from $R^{(1)}, \ldots, R^{(q)}$ and obtain
  $R'^{(1)}, \ldots, R'^{(q)} \subseteq [-N, N]^{q-1}$ such that
  $R'^{(1)} \times \ldots \times R'^{(q)}$ contains no feasible arrangements
  (understood either as elements of $\mathbb{Z}_{N'}^{q(q-1)}$ or
  $\mathbb{Z}^{q(q-1)}$).
  Finally, due to Lemma~\ref{lem:clt-counts}.3, we can remove
  at most $\frac{\mu}{2}q^n$ elements from the sets $S^{(1)}, \ldots, S^{(q)}$
  to obtain symmetric sets $T^{(1)}, \ldots, T^{(q)}$ such that, by
  Lemma~\ref{lem:clt-counts}.1, the product
  $T^{(1)} \times \ldots \times T^{(q)}$ contains no restricted progressions.
\end{proof}

It remains to prove Lemma~\ref{lem:clt-counts}. We achieve this by utilizing
Theorem~\ref{cor:lclt}.

\begin{proof}[Proof of Lemma~\ref{lem:clt-counts}]
  Point 1 is just a restatement of Claim~\ref{cl:progression-feasible}.

  We turn to Point 3 next. Consider $X \in \mathbb{Z}_q^n$ sampled uniformly at
  random. Recall our notation
  $W(x) = (W_1(x),\ldots,W_{q-1}(x)) = (w_1(x)-2\lfloor n/2q\rfloor,\ldots,
  w_{q-1}(x)-2\lfloor n/2q\rfloor)$ for $x\in\mathbb{Z}_q^n$ and
  the random variable $W = W(X)$. Clearly, we can apply~\eqref{eq:23}
  to $W$ and obtain
  \begin{align*}
    \frac{1}{Cn^{(q-1)/2}}\le
    \Pr[W = w] = \frac{\# w}{q^n}
    \le \frac{C}{n^{(q-1)/2}}\;,
  \end{align*}
  which yields the conclusion after rearranging the terms.

  As for Point~$2$, consider a choice of uniform random
  restricted progression $X^{(1)}, \ldots, X^{(q)}$. 
  We will apply Theorem~\ref{cor:lclt} to random variables
  \begin{align*}
    M_{\same}(a)
    &:= \left|\left\{ i\in[n]: X_i^{(1)}, \ldots, X_i^{(q)}
      = a, a, \ldots, a
      \right\}\right|-\left\lfloor\frac{n}{2q}\right\rfloor\;,\\
    M_{\cycle}(a)
    &:= \left|\left\{ i\in[n]: X_i^{(1)}, \ldots, X_i^{(q)}
      = a, a+1, \ldots, a-1
      \right\}\right|-\left\lfloor\frac{n}{2q}\right\rfloor\;.
  \end{align*}
  We let $M := \left(M_{\same}(1), \ldots, M_{\same}(q-1), M_{\cycle}(0),
    \ldots, M_{\cycle}(q-1)\right)$ (note that $M \in \mathbb{Z}^{2q-1}$).
  Now we need to specify a relation between possible values of $M$ and
  feasible weight arrangements $w^{(1)}, \ldots, w^{(q)}$.
  Observe that each possible value $m$ of $M$ uniquely determines a
  feasible weight arrangement $w^{(1)}, \ldots, w^{(q)}$.
  It turns out that there is a reasonably simple characterization of
  the set of tuples $m$
  that give rise to a given arrangement $w^{(1)},\ldots,w^{(q)}$.
  Namely, we check that these values form a linear one-dimensional solution space
  with triangular structure given by
  \begin{align}
    m_{\cycle}(0)
    &=
      k\;,
    \label{eq:08}\\
    m_{\same}(a)
    &=
      w_a^{(2)}-m_{\cycle}(a-1)\;,\quad a=1,\ldots,q-1\;,
    \label{eq:09}\\
    m_{\cycle}(a)
    &=
      w_a^{(1)}-m_{\same}(a)\;,\quad\quad\ \ a=1,\ldots,q-1\;.
      \label{eq:10}
  \end{align}
  for every $k \in \mathbb{Z}$. Let us denote each solution given by
  \eqref{eq:08}-\eqref{eq:10} by $m(w^{(1)}, \ldots, w^{(q)}; k)$.
  Therefore, we have
  \begin{align}\label{eq:12}
    \#\left(w^{(1)},\ldots,w^{(q)}\right)=
    (2q)^n\sum_{k\in\mathbb{Z}}
    \Pr\left[M=m\left(w^{(1)},\ldots,w^{(q)};k\right)\right] \; .
  \end{align}
  To establish \eqref{eq:11} we will separately bound this sum from below
  and from above. For the lower bound, first observe that as long as
  $|k|\le N$, then also
  $|w^{(j)}_a| \le N$, we can bound absolute values of all elements
  of the tuple $m\left(w^{(1)},\ldots,w^{(q)};k\right)$ with
  \begin{align*}
    \left|m_{\cycle}(a)\right|, \left|m_{\same}(a)\right|\le
    2qN
  \end{align*}
  and consequently obtain bounds on the $2$-norm and
  use~\eqref{eq:23} to bound the
  probability in~\eqref{eq:12}:
  \begin{align*}
    \left\|m\left(w^{(1)},\ldots,w^{(q)};k\right)\right\|_2^2
    &\le 8q^3N^2\;,\\
    \Pr\left[M=m\left(w^{(1)},\ldots,w^{(q)};k\right)\right]
    &\ge
    \frac{1}{Cn^{(2q-1)/2}}\;.
  \end{align*}
  Consequently,
  \begin{align*}
    \#\left(w^{(1)},\ldots,w^{(q)}\right)
    &\ge
    (2q)^n\sum_{k \in [-N, N]}
      \Pr\left[M=m\left(w^{(1)},\ldots,w^{(q)};k\right)\right]\\
    &\ge \frac{(2q)^nN}{Cn^{(2q-1)/2}}
      \ge \frac{(2q)^n}{CN^{2(q-1)}} \;,
  \end{align*}
  and rearranging gives the lower bound in~\eqref{eq:11}.
  For the upper bound, first note that clearly
  \begin{align*}
    \left\|m\left(w^{(1)},\ldots,w^{(q)};k\right)\right\|_2^2\ge
    m_{\cycle}(0)^2=k^2\;.
  \end{align*}
  This time we need to use the more precise error bound
  from~\eqref{eq:26}. Continuing the computation,
  \begin{align*}
    &\frac{\#\left(w^{(1)},\ldots,w^{(q)}\right)}{(2q)^n}
    =
    \sum_{k \in \mathbb{Z}}
    \Pr\left[M=m\left(w^{(1)},\ldots,w^{(q)};k\right)\right]\\
    &\qquad\le\frac{1}{n^{(2q-1)/2}}\cdot \left(\sum_{k\in\mathbb{Z}}
    O\left(\exp\left(-\frac{k^2}{Cn}\right)\right)+
    o\left(\min\left(1,\frac{n}{k^2}\right)\right)
    \right)
    \\
    &\qquad\le\frac{1}{n^{(2q-1)/2}}\cdot\left(
    \sum_{D=0}^{\infty}\sum_{D\sqrt{n}\le k<(D+1)\sqrt{n}}
      O\left(\exp\left(-D^2/C\right)\right)
      +o\left(1/D^2\right)
      \right)
    \\
    &\qquad\le\frac{1}{n^{q-1}}\cdot\left(
      \sum_{D=0}^{\infty} O\left(\exp\left(-D^2/C\right)\right)
      +o\left(1/D^2\right)\right)
    \le \frac{C}{N^{2(q-1)}}\;,
  \end{align*}
  and another rearrangement of terms finishes the proof.
\end{proof}

\begin{remark}
  Strictly speaking, we did not need slightly more complicated upper
  bound in~\eqref{eq:11} to establish that Theorem~\ref{thm:feasible-removal}
  implies Theorem~\ref{thm:product-removal}. However, this upper bound
  allows us to reverse the reasoning and obtain also that
  Theorem~\ref{thm:product-removal} implies Theorem~\ref{thm:feasible-removal}.
  We omit the details, but the proof is a straightforward reversal of the
  ``forward'' argument.
\end{remark}

\subsection{Proof of Theorem~\ref{thm:feasible-removal}}

To prove Theorem~\ref{thm:feasible-removal} we will need the hypergraph
removal lemma originally used in the proof of Szemerédi's theorem.
To state the removal lemma we first define hypergraphs and simplices.

\begin{definition}
  A \emph{$k$-uniform hypergraph} is a pair $H = (V, E)$, where $E$
  is a set of subsets of size $k$ (\emph{edges}) of a finite set
  of \emph{vertices} $V$. A \emph{$k$-simplex} is the unique $k$-uniform
  hypergraph with $k+1$ vertices and $k+1$ edges.
\end{definition}

\begin{theorem}[\cite{RS04, NRS06, Gow07}]
  \label{thm:hypergraph-removal}
  For every $k\ge 2$ and every $\eps > 0$ there exists
  $\delta := \delta(k,\eps) > 0$ such that for all $k$-uniform hypergraphs
  $H$ with $N$ vertices: If $H$ contains at most $\delta N^{k+1}$ simplices,
  then it is possible to remove at most $\eps N^{k}$ edges from $H$
  and obtain a hypergraph that does not contain any simplices.
\end{theorem}

Note that a $2$-uniform hypergraph is a simple graph, a $2$-simplex is a
triangle and Theorem~\ref{thm:hypergraph-removal} restricted to $k=2$ is
the triangle removal lemma.
With Theorem~\ref{thm:hypergraph-removal} we are ready to prove the removal
property  for feasible arrangements.

Let $R^{(1)}, \ldots, R^{(q)} \subseteq \mathbb{Z}_N^{q-1}$.
We define a $(q-1)$-uniform hypergraph
$H = (X, E)$. The set of vertices of $H$ consists of $q$ disjoint parts
$X = X^{(1)}\cup\ldots\cup X^{(q)}$ with each of the parts indexed by
$\mathbb{Z}_N^{q-1}$. Therefore, $H$ has $q N^{q-1}$ vertices.

The edge set also consists of $q$ disjoint parts
$E = E^{(1)} \cup \ldots \cup E^{(q)}$ such that
\begin{align*}
  E^{(j)}\subseteq X^{(1)}\times\ldots\times X^{(j-1)}\times X^{(j+1)}\times
  \ldots\times X^{(q)}\;.
\end{align*}
Therefore, every simplex in $H$ must contain one vertex from each $X^{(j)}$ and
one edge from each $E^{(j)}$.

Recall that a vertex $x^{(j)} \in X^{(j)}$ is of the form
$x^{(j)}=\left(x_1^{(j)}, \ldots, x_{q-1}^{(j)}\right)\in\mathbb{Z}_N^{q-1}$.
To define the edges of $H$ it will be useful to let
$x_0^{(j)} := -\sum_{a=1}^{q-1} x_a^{(j)}$. With that in mind, we say
that
\begin{align*}
  \left(x^{(1)}, \ldots, x^{(j-1)}, x^{(j+1)},\ldots,x^{(q)}\right)\in E^{(j)}
  \iff
  \left(w_1^{(j)}, \ldots, w_{q-1}^{(j)}\right)\in R^{(j)}\;,
\end{align*}
where
\begin{align}
  \label{eq:13}
  w_a^{(j)} := 
  \sum_{t=1}^{q-1} \sum_{b=a-t+1}^a x_b^{(j-t)}\;.
\end{align}
In the expression above the indices $b$ and $j-t$ are understood to
``wrap around'' modulo $q$.
To clarify by example (which might be useful to keep in mind throughout
the proof), for $q=4$, $j=2$ and $a=1$ we get
\begin{align*}
  w_1^{(2)} = x_1^{(1)} + x_0^{(4)}+x_1^{(4)} + x_3^{(3)}+x_0^{(3)}+x_1^{(3)}\;.
\end{align*}
We use the tuple arrangement
$\left(w_1^{(j)},\ldots,w_{q-1}^{(j)}\right)\in R^{(j)}$
defined in \eqref{eq:13} as a label of the edge
$\big(x^{(1)},\ldots,\allowbreak x^{(j-1)},\allowbreak x^{(j+1)},
  \allowbreak \ldots,x^{(q)}\big)\in E^{(j)}$.

We now proceed to checking that simplices in $H$ and feasible arrangements
in $R^{(1)}\times\ldots\times R^{(q)}$ correspond to each other.
To that end we start with some preparation. First, observing that
by definition $\sum_{a=0}^{q-1} x_a^{(j)} = 0$, we can rewrite \eqref{eq:13}
as
\begin{align}
  w_a^{(j)} = \sum_{t=0}^{q-1}\sum_{b=a-t+1}^a x_b^{(j-t)}\;,
  \qquad a=1,\ldots,q-1,\ j=1,\ldots,q\;.
  \label{eq:16}
\end{align}
Let $w_0^{(j)} := -\sum_{a=1}^{q-1} w_a^{(j)}$ and
check that \eqref{eq:16} naturally extends to
\begin{align*}
  w_0^{(j)} = \sum_{t=0}^{q-1}\sum_{b=-t+1}^0x_b^{(j-t)}\;.
\end{align*}

We now make two claims going from simplices to feasible arrangements and
vice versa.
\begin{claim}
  \label{cl:simplex-feasible}
If some $q$ vertices
$x^{(1)}, \ldots, x^{(q)}$ of the hypergraph $H$ form a simplex, then
the corresponding arrangement formed by edge labels
$\left(w^{(1)}, \ldots, w^{(q)}\right) \in
R^{(1)}\times\ldots\times R^{(q)}$ is feasible.
\end{claim}
\begin{proof}
The feasibility requirement from \eqref{eq:14} and
\eqref{eq:15} can be rewritten using $w_0^{(j)}$ as
\begin{align}
  \forall j=1,\ldots,q-2: \forall a=1,\ldots,q-1:
  w_a^{(j+2)}=-w_{a-1}^{(j)}+w_{a-1}^{(j+1)}+w_a^{(j+1)}\;.
  \label{eq:17}
\end{align}
To verify \eqref{eq:17} we compute (taking special care for $t\in\{0,q-1\}$
and still keeping in mind
$\sum_{a=0}^{q-1} x_a^{(j)}=0$)
\begin{align}
  -w_{a-1}^{(j)}+w_{a-1}^{(j+1)}+w_a^{(j+1)}
  &=\sum_{t=0}^{q-1}\left(
  -\sum_{b=a-t}^{a-1} x_b^{(j-t)}+\sum_{b=a-t-1}^{a-1}x_b^{(j-t)}
  +\sum_{b=a-t}^ax_b^{(j-t)}
    \right)
  \nonumber\\
  &=\sum_{t=0}^{q-1}\sum_{b=a-t-1}^a x_b^{(j-t)} =
    \sum_{t=0}^{q-1}\sum_{b=a-t+1}^a x_b^{(j+2-t)} = w_a^{(j+2)}\;.
    \label{eq:18}
\end{align}
\end{proof}

\begin{claim}
  \label{cl:feasible-simplex}
  For every feasible arrangement
  $w^{(1)}, \ldots, w^{(q)}\in R^{(1)}\times\cdots\times R^{(q)}$ there
  exist exactly $N^{(q-1)(q-2)}$ simplices in $H$ labeled with
  $w^{(1)},\ldots,w^{(q)}$. Furthermore, these simplices are edge disjoint.
\end{claim}
\begin{proof}
  We will show that every $x^{(2)}, \ldots, x^{(q-1)} \in \mathbb{Z}_N^{(q-1)(q-2)}$
  can be extended to a simplex $x^{(1)},\ldots,x^{(q)}$ labeled with the
  feasible arrangement $w^{(1)},\ldots,w^{(q)}\in
  R^{(1)}\times\ldots\times R^{(q)}$. First, by inspection we see
  that for fixed $x^{(2)},\ldots,x^{(q-1)},w^{(1)},\ldots,w^{(q)}$ the value of
  $x^{(1)}$ can be determined from the formula for $w^{(2)}$ given by
  \eqref{eq:13}. Similarly, the value of $x^{(2)}$ can be determined from
  \eqref{eq:13} for $w^{(1)}$.

  We still need to check that the vertices $x^{(1)},\ldots,x^{(q)}$ obtained
  in this way satisfy \eqref{eq:13} for $j=3,\ldots,q$. But this follows
  by induction, using \eqref{eq:17} (recall that arrangement
  $w^{(1)},\ldots,w^{(q)}$ is feasible) and a rearrangement of the computation
  in \eqref{eq:18}.

  To argue that the simplices are edge disjoint,
  we first observe that two different simplices $x^{(1)}, \ldots, x^{(q)}$
  and $y^{(1)}, \ldots, y^{(q)}$ with the same label $w^{(1)},\ldots,w^{(q)}$
  have to differ on at least two vertices. This is because for
  two simplices $x^{(1)},\ldots,\allowbreak x^{(j-1)},x,\allowbreak x^{(j+1)},\ldots,x^{(q)}$
  and $x^{(1)},\ldots,x^{(j-1)},y,x^{(j+1)},\ldots,x^{(q)}$ with $x_a \ne y_a$,
  the formula \eqref{eq:13} implies that
  \begin{align*}
    w_a^{(j+1)}\left(x^{(1)},\ldots,x,\ldots,x^{(q)}\right)
    \ne w_a^{(j+1)}\left(x^{(1)},\ldots,y,\ldots,x^{(q)}\right)\;.
  \end{align*}
  Therefore, any two ($q-1$)-hyperedges of simplices $x^{(1)}, \ldots, x^{(q)}$
  and $y^{(1)}, \ldots, y^{(q)}$ with the same label must differ on at least one
  vertex.
\end{proof}

With the two claims the proof is almost finished: For $\mu > 0$
we let $\mu':= \mu/q^{q+1}$ and take $\delta := \delta(q-1, \mu')$
from the hypergraph removal lemma.
Let $R^{(1)},\ldots,R^{(q)}\subseteq \mathbb{Z}^{q-1}_N$ be such that
the product $W^{(1)}\times\ldots\times W^{(q)}$ contains at most $\delta N^{2(q-1)}$
feasible arrangements. By Claims~\ref{cl:simplex-feasible}
and~\ref{cl:feasible-simplex}, the hypergraph $H$ contains at most
$\delta N^{q(q-1)} = \frac{\delta}{q^q}\cdot|X|^q \le \delta|X|^q$ simplices.
By the hypergraph removal lemma, it is possible to remove at most
$\mu' |X|^{q-1} = \mu' q^{q-1} N^{(q-1)^2}$ edges from $H$ to make
it simplex-free. Let $\tilde{E}$ be the set of removed edges.

We define
\begin{align*}
  Z^{(j)} := \left\{w^{(j)}\in R^{(j)}:
  \text{at least $N^{(q-1)(q-2)}/q$
  edges in $\tilde{E}$ are labeled with $w^{(j)}$}
  \right\}
\end{align*}
and let $V^{(j)} := R^{(j)}\setminus Z^{(j)}$.
Observe that $\left|Z^{(j)}\right| \le \mu'q^qN^{q-1} = \mu/q\cdot N^{q-1}$,
and therefore indeed the set of removed arrangements has total density
at most $\mu$.

We argue that
the product $V^{(1)}\times\cdots\times V^{(q)}$ does not contain a feasible
arrangement. Indeed, let $w^{(1)},\ldots,w^{(q)}$ be a feasible arrangement
in $R^{(1)}\times\cdots\times R^{(q)}$.
By Claim~\ref{cl:feasible-simplex},
the hypergraph $H$ contains $N^{(q-1)(q-2)}$ edge disjoint simplices labeled
with $w^{(1)},\ldots,w^{(q)}$.
Since those simplices disappear from $H$ after removing $\tilde{E}$,
each of them must intersect $\tilde{E}$ on at least one edge. By averaging,
there must exist $j$ such that $\tilde{E}$ contains at least
$N^{(q-1)(q-2)}/q$ edges labeled with $w^{(j)}$. But that implies that
$w^{(j)}$ was removed from $R^{(j)}$ and the arrangement
$w^{(1)},\ldots,w^{(q)}$ does not occur in $V^{(1)}\times\cdots\times V^{(q)}$.
\qed

\bibliographystyle{alpha}
\bibliography{bibliography}

\end{document}